\documentclass[a4paper,11pt, english]{article}
\usepackage[utf8x]{inputenc}
\usepackage{babel}
\usepackage[margin=1.25in]{geometry}
\usepackage{babel,amsmath,amssymb,amsthm,enumitem,setspace,authblk}
\usepackage[small,bf]{titlesec}
\titlelabel{\thetitle.\hspace{0.5em}}
\usepackage[colorlinks,allcolors=blue]{hyperref}
\usepackage{comment}

\numberwithin{equation}{section}

\newtheorem{theorem}[subsection]{Theorem}
\newtheorem{lemma}[subsection]{Lemma}

\newtheorem{proposition}[subsection]{Proposition}

\theoremstyle{definition}
\newtheorem{definition}[subsection]{Definition}
\newtheorem{remark}[subsection]{Remark}

\newcommand{\supp}{\text{supp}}

\begin{document}
\title{On a problem of optimal mixing}
\author[1,2]{K. O. Sokolov}%
\affil[1]{Lomonosov Moscow State University}
\affil[2]{"Vega" Institute}
\date{}

\onehalfspacing
\maketitle

\begin{abstract}

We consider the simultaneous optimal transportation of measures, where the target marginal is not necessarily fixed. For this problem, we prove the existence of a solution for completely regular spaces and investigate the structure of the discrete problem. We establish a connection between the Monge problem and the Kantorovich problem by showing that their functionals are equal and that the solutions coincide in Euclidean space.

\end{abstract}

\section{Introduction}

   Recently, various modifications to the optimal transportation problem have been actively researched (\cite{Beiglbck2016}, \cite{Bogachev2022}, \cite{2012BogachKoles}, \cite{Gozlan2017}). Recall the classical definition of the Monge-Kantorovich problem. We have two topological spaces, $X$ and $Y$, with two probability measures, $\mu$ and $\nu$. We also have a function $c$ that is defined on the product space $X\times Y$. This function is called the cost function.
   The goal is to minimize
    $$\int c(x,y) \; d\pi(dx,dy),
    $$
   where the infimum is taken by all measures $\pi\in\Pi(\mu,\nu)$ with fixed projections, that is, $\pi(A\times Y)=\mu(A)$ and $\pi(X\times B)=\nu(B)$ for all Borel $A$ and $B$.
 
    In this paper, we consider a modification of the classical optimal transport problem. Unlike the classical problem, we consider several probability measures $\{\mu_i\}_{i=1}^n$ on the space $X$ and we need to transform them into a single fixed or variable measure $\nu$ using a plan or mapping. 

    The problem of simultaneous optimal transport appeared in the literature earlier in the following formulation. Let
    \begin{equation}
			\Pi(\overline{\mu},\overline{\nu}) := \{ \pi \in \Pi(\mu,\nu) \; : \;  \int_X \pi^x(B) d\mu_i = \nu_i(B)\;  \forall\;  1\leq i\leq n, \; B\in\mathcal{B}(Y) \},
    \end{equation}
    where $\mu := \frac{\mu_1+\dots+\mu_n}{n}$ and $\nu := \frac{\nu_1+\dots+\nu_n}{n}$. 
    The goal is to minimize 
    \begin{equation}
            \inf_{\pi}\left\{ \int c(x,y) \; d\pi \; : \; \pi \in \Pi(\overline{\mu},\overline{\nu})\right\}
    \end{equation}

    Here, we need to transform each measure $\mu_i$ to $\nu_i$ using a transport plan $\pi$.
    
    In a short paper \cite{Wolansky2020}, the vector-valued Kantorovich problem is formulated for compact spaces, and a dual theorem is presented for this problem.

    In a recent paper \cite{Wang2022SimultaneousOT}, a dual theorem and a theorem on the equality of minima in the Monge and Kantorovich problems were proved. However, in order to ensure the existence of at least one suitable mapping for the measures $\{\mu_i\}_{i=1}^n$, an additional jointly atomless condition was imposed, which significantly restricts the choices of these measures. 

    In section 2, we formulate the problems, give a more constructive proof of Lyapunov's theorem, and prove the existence of solutions for problems \ref{kantorovich} and \ref{barycentrproblem}. In section 3, we consider a discrete problem and pay special attention to the problem with an unfixed target marginal. In section 4, we prove the theorem on the equality of the minimum and infimum in the Kantorovich and Monge problems, respectively. In section 5, we prove the connection between solutions to simultaneous transport problems for the statements of Monge and Kantorovich.


\section{Problem statements and the existence of solutions}

    The possibility of formulating the Monge problem of simultaneous optimal transport is provided by Lyapunov's theorem, which guarantees the existence of maps that transform several atomless measures into a fixed one. In the following lemma \ref{lemmaLyap}, a simpler and more constructive approach to constructing such a map is presented, which allows us to provide a specific example of this type of mapping on the real line.

    \begin{lemma}\label{lemmaLyap}
        Given an absolutely continuous measure $\nu$ on $[0, 1]$, there exists a Borel transformation T of $[0, 1]$ that preserves Lebesgue measure $\lambda$ and takes the measure $\nu$ to $\nu([0,1])\lambda$.
    \end{lemma}
    \begin{proof}
        Consider the signed measure $\eta := \nu - \lambda$ with density $\rho_\eta$ and define the sets $$A^+:=\rho_\eta^{-1}((0,+\infty)), \quad A^- := \rho_\eta^{-1}((-\infty, 0)) \quad A^0:=[0,1]\setminus (A^+\cup A^-).$$
        The restriction of the measure $\eta$ to the sets $A^+$ and $A^-$ is denoted as $\eta^+$ and $\eta^-$, respectively. 
    
        Let 
        $$
        B^+(t) := \{ x\in A^+ : \eta^+([0,x]\cap A^+) = t\}
        $$ 
        and
        $$B^-(t) := \{ x\in A^- : \eta^-([0,x]\cap A^-) = t\}.$$
        For any $r\in[0,1-\lambda(A^0)]$ we define  
        $$ C(r) := \left\{B^+(t)\cup B^-(t) \; : \; \lambda\left(\bigcup_{s\in[0,t]} (B^+(s)\cup B^-(s)\right) = r\right\}.
        $$
        Now, let us define the desired mapping
        $$T(x)=
            \begin{cases}
             r, \quad x\in C(r),
            \\
            \lambda(A^+\cup A^-)+\lambda([0,x]\cap A^0), \quad x\in A^0.
            \end{cases}
        $$
        \end{proof}
        \begin{remark}
        The proof provides an explicit construction, which, in the case of $n=2$ and for a quadratic function on a straight line, gives a reasonable option for an optimal mapping that is similar to the classical solution to the optimal transport problem on a straight line. However, as shown below, this mapping is not necessarily optimal. 
        \end{remark}

        Let us formulate Lyapunov's theorem, the proof of which can be found in \cite{Bogachev2007}.

    \begin{theorem}[Lyapunov]
        Let $\mu_1,\dots,\mu_n$ be atomless Borel probability measures on a Souslin space $X$. Then, for every Borel probability measure $\nu$ on $X$ there exists a Borel transformation $T: X\to X$ such that $\mu_i\circ T^{-1}=\nu$ for all $1\leq i\leq n$.  
    \end{theorem}

    After we have established that there is a mapping for any finite set of measures that transform them into a fixed one, the next natural question is about what the optimal mapping would look like in terms of minimizing some functional. The problem in Monge's formulation aims to answer this question. 
    
    Let $\mu_1,..., \mu_n$ and $\nu$ be atomless Borel probability measures on topological spaces  $X$ and $Y$ respectively, let $c: X\times Y \to [0,+\infty]$ be lower semi-continuous. Let us define the set of functions 
			$${\mathcal{T}(\overline{\mu},\nu)=\{ T: X\to Y \;|\; \mu_i \circ T^{-1} =\nu, \;\; 1\leq i\leq n\}}.$$ The goal is to find
			\begin{equation} \label{Monge}
				\inf \left\{ \sum_{i=1}^n\int_{X}c(x,T(x))\,d\mu_i : T\in\mathcal{T}(\overline{\mu},\nu) \right\}.
			\end{equation}
    

    The solution to the Monge problem may not always exist, even in its classical formulation. Therefore, we will explore the Kantorovich problem and examine the relationship between these two problems. 

    Let $\mu_1,..., \mu_n$ and $\nu$ be atomless Borel probability measures on topological spaces  $X$ and $Y$ respectively, let $c: X\times Y \to [0,+\infty]$ be lower semi-continuous.  Let us define the set
        \begin{equation}\label{plansfixmarg}
			\Pi(\overrightarrow{\mu},\nu) := \{ \pi \in \Pi(\mu,\nu) \; : \;  \int_X \pi^x(B) d\mu_i = \nu(B)\;  \forall\;  1\leq i\leq n, \; B\in\mathcal{B}(Y) \},
	\end{equation}
    where $\mu := \frac{\mu_1+\dots+\mu_n}{n}$. The goal is to find
        \begin{equation}\label{kantorovich}
            \inf_{\pi}\left\{ \int c(x,y) \; d\pi \; : \; \pi \in \Pi(\overrightarrow{\mu},\nu)\right\}
        \end{equation} and optimal transport plan.

    \begin{remark}
        It's easy to see that the set $\Pi(\overrightarrow{\mu},\nu)$ is not empty. 
        In addition to Lyapunov's theorem guaranteeing this, it also follows from if we take $\pi^x \equiv \nu$ for all $x$, then such a $\pi = \nu \otimes\mu$ would belong to the set $\Pi(\overrightarrow{\mu},\nu)$.
    \end{remark}

    \begin{remark}\label{linearconst}
        Note that this problem can be rephrased as a problem with linear constraints \cite{Zaev2015OnTM}.

        \begin{equation}
 			\inf\left\{  \int c(x,y)\; d\pi : \; \pi\in\Pi(\mu,\nu), \int f(x,y)\;d\pi =0, \;\; \forall f\in\mathcal{F} \right\},
 		\end{equation}
 		where
 		\begin{equation}
 			\mathcal{F}= \left\{ \psi(y)\cdot\left(\frac{d\mu_i}{d\mu}(x)-\frac{d\mu_j}{d\mu}(x)\right):\;\ i,j\in\{1, ... ,n\},\;\; \psi(y)\in C(Y)\right\}
 		\end{equation}

   Or simply, 
   \begin{equation}\label{conditionofproblem}
       \int \frac{d\mu_i}{d\mu}\; d\pi^y = 1, \quad \nu\text{-a.s.}
    \end{equation}
    \end{remark}

    \begin{theorem}
        Let $X$ and $Y$ be completely regular topological spaces, let $\mu_1,\dots,\mu_n,\nu$ be Radon measures, let cost function  $c:X\times Y\to [0,+\infty)$ be lower semi-continuous, let $\frac{d\mu_i}{d\mu}$ be continuous $\mu$-a.s., then infimum in (\ref{kantorovich}) is attained, that is, there is optimal plan $\pi\in\Pi(\overrightarrow{\mu},\nu)$.
    \end{theorem}
    \begin{proof}
        Let us now prove that the set $\Pi(\overrightarrow{\mu},\nu)$ is compact in the weak topology.
        To do this, note that $\Pi(\overrightarrow{\mu},\nu)$, among other things, has fixed projections, which means, as is well known, it is uniformly tight.

        Let us show that $\Pi(\overrightarrow{\mu},\nu)$ is closed.
        Let  $\pi_n\in \Pi(\overrightarrow{\mu},\nu)$  be sequence of measures converges to some $\pi$. To do this, we can use proposition 4.3.17 \cite{Bogachev2018}, which means that $\frac{d\mu_i}{d\mu}\pi_n$ converges weakly to $\frac{d\mu_i}{d\mu}\pi$. Hence projection of $\frac{d\mu_i}{d\mu}\pi$ on $Y$ is equal to $\nu$. Therefore, by Prokhorov's theorem $\Pi(\overrightarrow{\mu},\nu)$ is compact in weak topology.

        Considering that for a semi-continuous function $c(x, y)$, the functional $\pi \mapsto \int c d\pi$ is also semi-continuous, we can conclude that the infimum in problem (\ref{kantorovich}) is attained.
    \end{proof}

    Assuming that the functions $\frac{d\mu_i}{d\mu}$ are continuous, we can derive a dual theorem. The proof of this theorem is presented in \cite{Wang2022SimultaneousOT}.

    \begin{theorem}
		Let $X$, $Y$ be Polish spaces, let $\mu_1,\dots,\mu_n$ and $\nu$ be atomless probability measures on the corresponding spaces, and let $c:X\times Y \to [0,\infty)$ be lower semi-continuous function, then
		$$ \inf\limits_{\pi\in	\Pi(\overrightarrow{\mu},{\nu})} K_c(\pi):= \int\limits_{X\times Y}c(x,y)\; d\pi = \sup\limits_{(\overrightarrow\phi,\psi)\in J(\overrightarrow{\mu},{\nu})} \int_X\phi(x) d{\mu} +  \int_Y\psi(y) d{\nu}, $$
		where 
		$$ J(\overrightarrow{\mu},{\nu}) = \{\phi\in C(X), \overrightarrow\psi\in C^n(Y): \phi(x)+ \frac{d\overrightarrow\mu}{d\mu}\cdot\overrightarrow\psi(y)\leq c(x,y) \}, \quad \psi(y):= \sum_i \psi_i(y).$$
    \end{theorem}


    Of particular interest is the situation where the marginal $\nu$ is not fixed. Specifically, let us consider a set of atomless Borel probability measures $\mu_1, \dots, \mu_n$ defined on the Euclidean space $\mathbb{R}^d$. The problem is to find  
        \begin{equation}\label{barycentrproblem}
            \inf_{\pi}\left\{ \int ||x-y||^2 \; d\pi \; : \; \pi \in \Pi(\overrightarrow{\mu})\right\},
        \end{equation}
        where $\Pi(\overrightarrow{\mu})$ is the set of plans   $\pi$ with fixed first marginal $\mu$, and the equality is true
        $$
            \int \pi^x(B) d\mu_i =\int \pi^x(B) d\mu_j, \text{ for any }B\in\mathcal{B}(\mathbb{R}^d), \; i,j\in\{1,\dots,n\}
        $$
        Note that if the supports of the measures $\mu_i$ are disjoint, then this problem is equivalent to finding the barycanter of the aforementioned measures. We will focus on the case where the measures $\mu_i$ share a common support.
    \begin{proposition}
        The infimum in problem \ref{barycentrproblem} is attained, where $\frac{d\mu_i}{d\mu}$ are continuous $\mu$-a.s. 
    \end{proposition}
    \begin{proof}

        Consider the minimizing sequence $\pi_n$. Note that the sequence of marginal distributions $\nu_n$ corresponding to the plans $\pi_n$ is uniformly tight. Indeed, if this were not the case, and for any compact set in $\mathbb{R}^d$, there was a subsequence $\nu_{n_k}$ with mass $\varepsilon$ outside of it, then by choosing $k$ large enough, the value of the cost functional could be made arbitrarily large, contradicting the fact that $\pi_n$ minimizes this functional. 

        Since the sequence $\nu_n$ is uniformly tight, the set of plans $\pi_n$ is also uniformly tight, since for any $\varepsilon>0$ one can choose the compacts $K_1$ and $K_2$ so that $\mu(K_1)>1-\varepsilon$ and $\nu_n(K_2)>1-\varepsilon$, which means $\pi_n(K_1\times K_2)>1-2\varepsilon$ for any $n$. So, according to Prokhorov's theorem, we can choose a weakly convergent subsequence $\pi_{n_k}$ to some plan $\pi^*$. Using the closure of the set $\Pi(\overrightarrow{\mu})$, which follows from $\mu$-a.s. of continuity $\frac{d\mu_i}{d\mu}$, we obtain that $\pi^*$ is the solution to the problem \ref{barycentrproblem}.
        
    \end{proof}
    

\section{Optimal mixing on discrete domain}
    Assume that we have a finite set $X=\{x_1,\dots, x_m\}\subset\mathbb{R}^d$, which is the support for each probability measure $\mu_i\in\mathcal{P}(X)$, where $1\leq i \leq n$.
    
    To begin with, let us note that if $\nu$ is a fixed discrete measure on the set $Y = \{y_1, \dots, y_l\} \subset \mathbb{R}^d$, then the problem of finding the optimal plan can be formulated as a linear programming problem. This follows from the fact that the simultaneous transport problem is a problem with linear constraints. By adding the corresponding condition, we can obtain a linear programming problem
    $$\langle\pi_{kj},c_{kj}\rangle \to \min, \quad
    \sum_i^m \pi_{kj} = \nu^j, \quad
    \sum_j^l \pi_{kj} = \mu^k, \quad
    \sum_k \langle \pi_{kj},\frac{\mu_i^k}{\mu^k} \rangle = \nu^j.
    $$

    Next, assume that $c(x,y)=||x-y||^2$. Consider the case when the measure $\nu$ is not fixed. To solve such a problem, we need to find the support of the marginal $\nu$.
    
    Let $\psi: X \to [0,1]^n$ be function such that
    \begin{equation}
        \psi : x \mapsto \frac{1}{n}\left(\frac{d\mu_1}{d\mu}(x),\dots,\frac{d\mu_n}{d\mu}(x)\right)
    \end{equation}

    Note that all points in the image are located within the $(n-1)$-dimensional simplex $\Delta_n$. Let $P$ be the image of the measure $\mu$ after applying the mapping $\psi$, and note that the barycenter of this measure coincides with the center of $\Delta_n$, since
    $$
    \int z \; dP(z) = \int  \frac{1}{n}\left(\frac{d\mu_1}{d\mu}(x),\dots,\frac{d\mu_n}{d\mu}(x)\right)\; d\mu = \left(\frac{1}{n},\dots,\frac{1}{n}\right)
    $$

    The following lemma is a well-known fact in convex geometry, and it follows from the fact that in an $(n-1)$-dimensional space, the convex hull of a set of points is the union of all possible simplices that can be constructed from any subset of $n$ points in that set.
 
    \begin{lemma}\label{lemma_1}
   There is a set of at most $n$ points from the support of  $P$ such that their convex hull contains the center of the simplex.
    \end{lemma}

        \begin{lemma}

        The optimal mixing points for the measures $\mu_1,\dots,\mu_n$ on $X$ are the barycenters of minimal subsets $\{x_{i_k}\}^l_{k=1}$ such that $l\leq n$, and the convex hull of the image of which, when transformed by $\psi$, contains the center of the simplex $\Delta_n$, and the weights on the points $\{x_{i_k}\}^l_{k=1}$ are set according to the corresponding weights on $\{\psi(x_{i_k})\}^l_{k=1}$, so that the barycenter of these points is the center of the simplex $\Delta_n$.

        \end{lemma}
        \begin{proof}

        Let $\pi$ be the optimal plan. Take any point $y\in\supp(\nu)$, where $\nu$ is the projection of $\pi$ on $Y$. Now consider the support $\{x_{i_k}\}^l_{k=1}$ of the conditional measure $\pi^y$. It is clear from the condition \ref{conditionofproblem} that the convex hull $\{\psi(x_{i_k})\}^l_{k=1}$ contains the center of the simplex $\Delta_n$. If the set $\{\psi(x_{i_k})\}^l_{k=1}$ is minimal (i.e. one from which it is impossible to remove one point without losing the property that the convex hull $\{\psi(x_{i_k})\}^l_{k=1}$ contains the center of the simplex), then the weights are set unambiguously and $y$ must be the barycenter $\{x_{i_k}\}^l_{k=1}$ with appropriate weights, since it is the barycenter that minimizes the functional 
        $$\mathcal{F}(y)=\int x \; d\pi^y$$.

        If the set is not minimal, then we divide it into minimal ones. And according to the lemma \ref{lemma_1}, the minimum set cannot consist of more than $n$ points. After that, we will replace the point $y$ with the barycenters of these sets, which obviously will not increase cost functional.

        \end{proof}

        Based on the above reasoning, we can more concretely formulate the problem, taking into account that from the subsets $\{x_{i_k}\}_{k=1}^l$ we can construct the support of optimal $\nu$, points of which are barycenters of these subsets. Specifically, we reduce this problem to a linear programming problem.

        That is, we need to find a measure $\nu$ on the discrete space $Y$ that, for known $\pi_y, \mu$, minimizes the integral
        $$\int ||x-y||^2\; d\pi^y d\nu,$$
        where $\pi\in\Pi(\mu,\nu)$.


\section{The equality of values in the Monge and Kantorovich problems}

In this section, we will discuss the theorem on the equality of minimum and infimum in the Kantorovich and Monge problems respectively. The proof of this theorem is based on ideas from  \cite{Bogachev2019, Pratelli2007OnTE}, where an analogous theorems was proved for a single source marginal. In the paper \cite{Wang2022SimultaneousOT} was proved an analogue of this theorem for the simultaneous transport problem for compact Polish spaces. We will provide an alternative proof for the more general case.

    To prove the main theorem, we need to use the following lemma, which is a simple consequence of Lyapunov's theorem \cite{Lyapunov1940}. 

        \begin{lemma}[Lyapunov]\label{lyapunov2}

        Let $\mu_1\dots,\mu_n$ be atomless non-negative measures on measurable space. Then the following sets are the same
        $$\{ (\mu_1(E),\dots,\mu_n(E))\; : \; E\in\mathcal{F}\} = \left\{ \left(\int g(x) \; d\mu_1\dots,\int g(x) \; d\mu_n\right)\; : \; 0\leq g(x)\leq 1 \right\},
        $$
        where $g$ is a $\mathcal{F}$-measurable function.
            
        \end{lemma}

        \begin{theorem}

        Let $X$ and $Y$ be compact Souslin spaces and let cost function  $c: X\times Y \to [0,+\infty)$ be continuous. Let $\mu_1,\dots,\mu_n,$ and $\nu$ be atomless probability measures on $X$ and $Y$ respectively, then 
        \begin{equation}\label{mininf}
        \min_{\pi\in\Pi(\overline{\mu},\nu)} \int c(x,y)\; d\pi = \inf_{T\in\mathcal{T}(\overline{\mu},\nu)} \int c(x,T(x))\; d\mu
        \end{equation}
            
        \end{theorem}

        \begin{proof}
            Let us fix  $\varepsilon>0$.
            Let $\pi$ be optimal plan.
            Let us choose an open cover of the space $X\times Y$ by sets $A'_m \times B'_m$ such that 
            $$
            \sup_{(x,y),(x',y')\in A'_m\times B'_m}|c(x,y)-c(x',y')|\leq \varepsilon,
            $$
            and choose a finite subcover. Using the sets $A'_m$, we will construct a disjoint set of sets $A_k$. We will then carry out similar constructions with the sets $B'_m$ to obtain sets of disjoint sets $\{A_k\}_{i=1}^{N}$ and $\{B_k\}_{j=1}^M$. These sets are defined on corresponding spaces and have certain properties:

            1) $\sup_{(x,y),(x',y')\in A_i\times B_j}|c(x,y)-c(x',y')|\leq \varepsilon;$

            2) $(A_{i_1}\times B_{j_1})\cap (A_{i_2}\times
            B_{j_2})= \varnothing$, where  $i_1\ne i_2$ or $j_1\ne j_2$;

            3) $\pi\left(\bigcup_{i,j}^{N,M} (A_i\times B_j)\right)=1.$ 

            Now we fix one of the sets $A_i$, where $i\in\{1,\dots N\}$, and consider the functions $g_k(x):=\pi^x(B_k)\leq 1$, where $\pi^x$ is a conditional measure at the point $x$. Using the \ref{lyapunov2} lemma inductively, we obtain a partition of the set $A_i$ on the set $\{X_i^k\}_{k=1}^M$ such that 
            $$
            \left(\mu_1(X_i^k),\dots,\mu_n(X_i^k)\right) = \left(\int g_k(x) \; d\mu_1\dots,\int g_k(x) \; d\mu_n\right).
            $$

            We will repeat this process for each $A_i$. Then, we define a mapping $T_i^k(x)$ on each $X_i^k$, which translates the constraints of measures $\mu_1, \dots, \mu_n$ to the constraint of measure $\nu$ on $B_k$ multiplied by a constant $\mu_j(X_i^k)$. It is easy to see that if we define $T:X\to Y$ on $X_i^k$ as $T_i^k(x)$, it will translate each $\mu_j$ to $\nu$, and

            $$
            \int c(x,T(x))\; d\mu \leq \sum_{i,k} \int c(x,T_i^k(x))\; d\mu \leq \sum_{i,k} \left( \int_{A_i\times B_k} c(x,y)\; d \pi + \varepsilon\pi(A_i\times B_k)\right) = 
            $$
            $$
            = \int c(x,y) \; d\pi + \varepsilon. 
            $$
            Given any choice of $\varepsilon$, we obtain the  equality (\ref{mininf}).

        \end{proof}

\section{The existence of an optimal mapping}

    In this section, we will be working in the space $X = \mathbb{R}^d$, with a quadratic cost function. The classical theory of optimal transport establishes a connection between the Monge and Kantorovich problems for this type of cost function using the concepts of cyclic monotonicity and the Rockafellar theorem. To recall these important ideas and concepts, let us first define them.

\begin{definition}

 We say that  $\Gamma\subset \mathbb{R}^d\times \mathbb{R}^d$ is cyclically monotone, if for any points $\{(x_i,y_i)\}_{i=1}^n\subset\Gamma$ following equality is true 
$$ \sum_{i=1}^n||x_i-y_i||^2 \leq  \sum_{i=1}^n ||x_{i+1}-y_i||^2,$$ where $x_{n+1}:=x_1$.

\end{definition}

\begin{theorem}[Rockafellar]

Assume that $\Gamma\subset \mathbb{R}^d\times \mathbb{R}^d$ is cyclically monotone. Then there exists a convex function $\phi: X\to \mathbb{R}\cup\{+\infty\}$ such that $\Gamma$ contained in graph of subdifferential of $\phi$.

\end{theorem}

\begin{lemma}
Let $\phi : \mathbb{R}^n \to \mathbb{R}\cup\{+\infty\}$ be convex function. Then, the set of points where the subdifferential of $\phi$ contains more than one element is Lebesgue negligible.
\end{lemma}

    In the case of problems with linear constraints, such as \ref{kantorovich}, the support of optimal plan is not cyclically monotone over the entire space.  However, in the situation where $\frac{d\mu_i}{d\mu}$ is simple functions, this property can be utilized, and it can be shown that in the region of constancy, the optimal plan will be determined by the map. To formally prove this, we need a variation lemma \cite{Zaev2015OnTM}.

        \begin{definition}
            Let $\alpha$ be non-negative measure on $X\times Y$. 
	   We say that $\alpha '$ is competitor of  
	   $\alpha$, if $\frac{d\mu_i}{d\mu}\cdot\alpha$ and       $\frac{d\mu_i}{d\mu}\cdot\alpha '$ has the same marginals for $i\in\{1,...,n\}$.
        \end{definition}

        \begin{definition}
        We say that $\Gamma \subset X\times Y$ for measurable function $c: X\times Y\to \mathbb{R}$ the set $\Gamma \subset X\times Y$ is called $c$-monotone, if for any $\alpha$ with finite support and $\text{supp}(\alpha)\subset \Gamma$ for any competitor $\alpha'$  
	$$ 
        \int c \; d\alpha \leq \int c \; d\alpha'
        $$
        \end{definition}

        \begin{definition}
		Plan $\pi\in\Pi(\overrightarrow{\mu},\nu)$ is called $c$-monotone, if there exists c-monotone $\Gamma$ such that $\pi(\Gamma)=1$.
	\end{definition}
    
    \begin{theorem}[Variational lemma]
		Let $X$ and $Y$  be Polish spaces, let function $c:X\times Y \to [0,\infty)$ be continuous, let $\mu_1,..., \mu_n,\nu$ be atomless Borel probability measures on $X$, and let $\frac{d\mu_i}{d\mu}$ be $\mu$-a.s. continuous. Let $\hat{\pi}$ be a solution of problem \ref{kantorovich}; then $\hat{\pi}$ is $c$-monotone plan.
    \end{theorem}

    Now, we will show that in a special case, we can relate the solutions of the Monge \ref{Monge} and Kantorovich \ref{kantorovich} problems.
        
    \begin{theorem}\label{forsimple}
            
    Let $X=Y=\mathbb{R}^d$  and $c(x,y)=||x-y||^2$. Assume that functions $\frac{d\mu_i}{d\mu}$ are simple and continuous $\lambda$-a.e., where $\lambda$ is Lebesgue measure. Then solution of the problem \ref{kantorovich} supported on a graph of some function, and the solution is unique w.r.t. $\lambda$.
            
    \end{theorem}
    \begin{proof}

        Indeed, let the measure $\pi$ be the solution to the problem \ref{kantorovich}. The space $X$ can be divided into sets $\{A_k\}_{k=1}^{\infty}$, where each $\frac{d\mu_i}{d\mu}$ is constant. The sets $\{A_k\}_{k=1}^{\infty}$ can be considered closed, since the functions $\frac{d\mu_i}{d\mu}$ are continuous $\lambda$-a.e. 

        Now, note that the support of $\pi$ on the set $A_k\times\mathbb{R}^d$ is cyclically monotone, since by virtue of the variational lemma of the competitor of the measure $\alpha'$ for the measure can be chosen as in the classical case. Since at constant $\frac{d\mu_i}{d\mu}$ while preserving marginals $\pi$, marginals of measures $\frac{d\mu_i}{d\mu}\pi$ are also preserved.
 
        It remains to use Rockafellar's theorem and the fact that the subdifferential of the convex function $\lambda$-a.e. consists of one element.
        Thus, $\pi$ is supported on the graph of some mapping and is the solution to the Kantorovich problem, so this mapping is the solution to the Monge problem.

    \end{proof}

    \begin{proposition}\label{monotoneforline}
        Let  $X=Y=\mathbb{R}$, then it follows from variational lemma that barycentric function $$g(y)=\int x \; d\pi^y$$ 
        is non-decreasing, where $\pi $ is solution of the problem \ref{kantorovich}. 
    \end{proposition}
    \begin{proof}

    Assume the contrary. Then there are $y_1 < y_2$ such that  
    $$
    g(y_1) = \int x \; d\pi^{y_1} > \int x \; d\pi^{y_2} = g(y_2).
    $$

    Let us take as $\alpha$ a discrete measure whose support is contained in the support $\pi$ and such that it is concentrated at points with the second coordinate $y_1$ or $y_2$ such that
    $$
    \int x \; d\alpha^{y_i} = \int x \; d\pi^{y_i},
    $$
    and 
    $$
    \int \frac{d\mu_k}{d\mu} \;d\alpha^{y_i} = 1,
    $$
    where $i\in\{1,2\}$, $k\in\{1,\dots,n\}$. It is easy to see that such a measure always exists. For simplicity, assume that $Y$-marginal of $\alpha$ is uniform. Now we can define its competitor $\alpha'$ such that it has the same $Y$-marginal, but the conditional measures are rearranged. 
    
    It remains to add that
    $$
    g(y_1)y_1 + g(y_2)y_2 >  g(y_2)y_1 + g(y_1)y_2,
    $$
    that is
    $$
    \int xy \; d\alpha' > \int xy \; d\alpha,
    $$
    hence
    $$
    \int |x-y|^2 \; d\alpha' < \int |x-y|^2 \; d\alpha.
    $$    
        
    \end{proof}
  
    Thus, using the proposition \ref{monotoneforline}  and the theorem \ref{forsimple}, one can explicitly construct a solution to the \ref{kantorovich} problem for simple $\frac{d\mu_i}{d\mu}$ on real line, each time choosing the smallest possible value of the barycentric function.

\end{document}